\def\DateTime{01/Jan/2018}
\def\Version{Version 1.0}
\def\yes{\if00}
\def\no{\if01}
\def\iftenpt{\no}
\def\ifelevenpt{\no}
\def\iftwelvept{\yes}

\iftenpt
\documentclass[leqno]{amsart}
\else\fi
\ifelevenpt
\documentclass[leqno,11pt]{amsart}
\else\fi
\iftwelvept
\documentclass[leqno,12pt]{amsart}
\else\fi

\usepackage{amssymb}
\usepackage{amscd}
\usepackage{verbatim,color}
\usepackage{mathrsfs,enumerate}

\ifelevenpt
\else\fi

\iftwelvept
\else\fi

\theoremstyle{plain}
\newtheorem{Theorem}{Theorem}[section]
\newtheorem{Proposition}[Theorem]{Proposition}
\newtheorem{Lemma}[Theorem]{Lemma}
\newtheorem{Corollary}[Theorem]{Corollary}
\newtheorem{Claim}{Claim}[Theorem]

\theoremstyle{definition}

\newtheorem{Example}[Theorem]{Example}

\renewcommand{\theTheorem}{\arabic{section}.\arabic{Theorem}}
\renewcommand{\theClaim}{\arabic{section}.\arabic{Theorem}.\arabic{Claim}}
\renewcommand{\theequation}{\arabic{section}.\arabic{equation}}

\newcommand{\ZZ}{{\mathbb{Z}}}

\newcommand{\RR}{{\mathbb{R}}}
\newcommand{\Spec}{\operatorname{Spec}}
\newcommand{\Supp}{\operatorname{Supp}}
\newcommand{\length}{\operatorname{length}}
\newcommand{\vol}{\operatorname{vol}}

\begin{document}

\title[Lower estimates of multiplicities]%
{Note on lower estimates of multiplicities}
\author{Atsushi Moriwaki}
\address{Department of Mathematics, Faculty of Science,
Kyoto University, Kyoto, 606-8502, Japan}
\email{moriwaki@math.kyoto-u.ac.jp}
\date{\DateTime, (\Version)}

\maketitle

\section*{Introduction}
The lower estimates of multiplicities are crucial to establish Faltings' product theorem (cf. \cite{Faltings}, \cite{Ever}, \cite{Ferr} and \cite{EdixEver}).
In this note we would like to give an explicit lower bound by using a lower saturated subset of $\RR_{\geq 0}^n$.

Let $A := k[X_1, \ldots, X_n]$ be the polynomial ring of $n$-variables over a field $k$, $\mathfrak{p}$ be a prime ideal of $A$,
and $x_i$ be the image of $X_i$ in $A/\mathfrak{p}$ for $i=1, \ldots, n$.
Let 
\[
K_i := \begin{cases}
k(x_i, \ldots, x_n) & \text{if $i=1, \ldots, n$},\\
k & \text{if $i=n+1$}
\end{cases}
\]
for $i=1, \ldots, n+1$, 
and $\Upsilon := \{ i \in \{ 1, \ldots, n \} \mid \text{$K_i$ is algebraic over $K_{i+1}$} \}$.
Let $\Delta$ be a lower saturated subset of $\RR_{\geq 0}^n$, that is, $\Delta$ is a non-empty subset of $\RR_{\geq 0}^n$ such that
$[0, \xi_1] \times \cdots \times [0, \xi_n] \subseteq \Delta$ for
all $(\xi_1, \ldots, \xi_n) \in \Delta$.
Here we define $\Sigma$, $(\RR^n)_\Upsilon$, $\Delta_\Upsilon$ and $\partial^{\gamma}$ ($\gamma = (e_1, \ldots, e_n) \in \ZZ_{\geq 0}^n$) as follows:
\[
\begin{cases}
\Sigma := \Delta \cap \ZZ_{\geq 0}^n, & (\RR^n)_\Upsilon := \{ (\xi_1, \ldots, \xi_n) \in \RR^n \mid \text{$\xi_j = 0$ ($\forall j \not\in \Upsilon$)} \}  \cong \RR^{\Upsilon},\\[1ex]
\Delta_\Upsilon := \Delta \cap (\RR^n)_\Upsilon, & \partial^{\gamma} 
:= \partial^{e_1 + \cdots + e_n}/\partial X_1^{e_1} \cdots \partial X_n^{e_n}
\end{cases}
\]
Then one has the following:

\begin{Proposition}[cf. Corollary~\ref{coro:lower:estimate:mult}]
We assume that the characteristic of $k$ is zero. 
Let $\mathfrak{I}$ be an ideal of $A$ such that $\mathfrak{p}$ is a minimal prime of $\mathfrak{I}$ and
$\partial^{\gamma}(\mathfrak{I}) \subseteq \mathfrak{p}$ for all $\gamma \in \Sigma$. Then
one has $\length_{A_\mathfrak{p}}((A/\mathfrak{I})_\mathfrak{p}) \geq \vol_{(\RR^n)_\Upsilon}(\Delta_\Upsilon)$, where 
$\vol_{(\RR^n)_\Upsilon}(\Delta_\Upsilon)$ is the volume of $\Delta_\Upsilon$ in $(\RR^n)_\Upsilon$.
\end{Proposition}

\section{Lower saturated subset}
For $\varphi = (\xi_1, \ldots, \xi_n) \in \RR_{\geq 0}^n$,
we set
$\mathbb{I}_\varphi := [0, \xi_1] \times \cdots \times [0, \xi_n]$.
A subset $\Delta$ of $\RR_{\geq 0}^n$ is said to be 
{\em lower saturated} if
$(0, \ldots, 0) \in \Delta$ and $\mathbb{I}_\varphi \subseteq \Delta$ for all $\varphi \in \Delta$.
Note that a lower saturated subset of $\RR_{\geq 0}^n$
is Lebesgue measurable (for details, see Proposition~\ref{prop:measurable:Delta}).
Similarly a subset $\Sigma$ of $\ZZ_{\geq 0}^n$ 
is said to be {\em lower
$\ZZ_{\geq 0}$-saturated} if $(0, \ldots, 0) \in \Sigma$ and
$\mathbb{I}_{\gamma} \cap \ZZ_{\geq 0}^n \subseteq \Sigma$ for all $\gamma \in \Sigma$.

\begin{Proposition}\label{prop:lower:saturated}
Let $\Sigma$ be a lower $\ZZ_{\geq 0}$-saturated subset of $\ZZ_{\geq 0}^n$. If we set
\[
\Delta(\Sigma) := \bigcup_{\gamma \in \Sigma} \mathbb{I}_{\gamma}\quad\text{and}\quad
\Delta'(\Sigma) := \kern-1em \bigcup_{\gamma = (e_1, \ldots, e_n) \in \Sigma} [e_1, e_1 + 1) \times \cdots \times [e_n, e_n + 1),
\]
then one has the following:
\begin{enumerate}[(1)]
\item
$\Delta(\Sigma)$ is a closed lower saturated subset of $\RR_{\geq 0}^n$
such that $\Delta(\Sigma) \cap \ZZ_{\geq 0}^n = \Sigma$.

\item
$\Delta'(\Sigma)$ is a lower saturated subset of $\RR_{\geq 0}^n$
such that $\Delta'(\Sigma) \cap \ZZ_{\geq 0}^n = \Sigma$.

\item
Let $\Delta$ be a lower saturated subset of $\RR_{\geq 0}^n$
such that $\Delta \cap \ZZ_{\geq 0}^n = \Sigma$. Then one has 
$\Delta(\Sigma) \subseteq \Delta \subseteq \Delta'(\Sigma)$. In particular,
$\vol(\Delta) \leq \#(\Sigma)$. 
\end{enumerate}
\end{Proposition}

\begin{proof}
In the followings, for $\xi \in \RR$, $\min\{ a \in \ZZ \mid \xi \leq a \}$ and
$\max \{ b \in \ZZ \mid b \leq \xi \}$ are denoted by $\lceil \xi \rceil$ and $\lfloor \xi \rfloor$, respectively.

\medskip
(1) The lower saturatedness of $\Delta(\Sigma)$ is obvious.
First let us see that $\Delta(\Sigma)$ is closed.
Let $\{ \varphi_m \}_{m=1}^{\infty}$ be a sequence of $\RR^n$ such that 
$\varphi_m \in \Delta(\Sigma)$ for all $m \geq 1$ and $\varphi := \lim_{m\to\infty} \varphi_m$ exists. If we set $\varphi_m = (\xi_{m1}, \ldots, \xi_{mn})$ and
$\gamma_m = (\lceil \xi_{m1}\rceil, \ldots, \lceil \xi_{mn} \rceil)$, then $\gamma_m \in \Sigma$. As $\lim_{m\to\infty} \varphi_m$ exists, there are $\gamma \in \Sigma$ and
a subsequence $\{ \varphi_{m_i} \}$ of $\{ \varphi_m \}$ such that
$\gamma_{m_i} = \gamma$ for all $i$. 
Then, as $\varphi_{m_i} \in \mathbb{I}_\gamma$, one has $\varphi = \lim_{i\to\infty} \varphi_{m_i} \in \mathbb{I}_\gamma$, so that
$\varphi \in \Delta(\Sigma)$.

Finally let us see that $\Delta(\Sigma) \cap \ZZ_{\geq 0}^n = \Sigma$. Indeed, if $\gamma \in \Delta(\Sigma) \cap \ZZ_{\geq 0}^n$,
then there is $\gamma' \in \Sigma$ such that $\gamma \in \mathbb{I}_{\gamma'}$, so that
$\gamma \in \mathbb{I}_{\gamma'} \cap \ZZ_{\geq 0}^n \subseteq \Sigma$.

\medskip
(2) Note that $\Delta'(\Sigma) = \{ (\xi_1, \ldots, \xi_n) \in \RR_{\geq 0}^n \mid (\lfloor \xi_1 \rfloor, \ldots, \lfloor \xi_n \rfloor) \in \Sigma \}$.
Therefore the lower saturatedness follows from the lower $\ZZ_{\geq 0}$-saturatedness of $\Sigma$.
The assertion $\Delta'(\Sigma) \cap \ZZ_{\geq 0}^n = \Sigma$ is obvious.

\medskip
(3) For $\gamma \in \Sigma$, one has
$\mathbb{I}_{\gamma} \subseteq \Delta$, so that the first assertion $\Delta(\Sigma) \subseteq \Delta$ follows.
Next let us see $\Delta \subseteq \Delta'(\Sigma)$. Indeed,
if we set $\gamma = (\lfloor \xi_1 \rfloor, \ldots, \lfloor \xi_n \rfloor)$ for $\varphi = (\xi_1, \ldots, \xi_n) \in \Delta$, then $\gamma \in \ZZ_{\geq 0}^n$ and $\gamma \in \mathbb{I}_\varphi$, so that
$\gamma \in \Delta \cap \ZZ_{\geq 0}^n = \Sigma$, and hence $\varphi \in \Delta'(\Sigma)$. 
The final assertion is obvious because $\vol(\Delta'(\Sigma)) = \#(\Sigma)$.
\end{proof}

\section{The key lemma}
Let $k$ be a field and
$A := k[X_1, \ldots, X_n]$ be the polynomial ring of $n$-variables over $k$. 
For $\gamma = (e_1, \ldots, e_n) \in \ZZ_{\geq 0}^n$,
we denote $e_1 + \cdots + e_n$ by $|\gamma|$ and the monomial $X_1^{e_1} \cdots X_n^{e_n}$
by $X^\gamma$. We set
$\Supp(f) := \{ \gamma \in \ZZ_{\geq 0}^n \mid a_{\gamma} \not= 0 \}$
for $f = \sum_{\gamma \in \ZZ_{\geq 0}^n} a_{\gamma} X^{\gamma} \in A$.
It is easy to see that
\begin{equation}\label{eqn:lem:key:estimate:01}
\Supp(f + g) \subseteq \Supp(f) \cup \Supp(g)
\end{equation}
for all $f, g \in A$.

\begin{Lemma}\label{lem:key:estimate}
Let $\mathfrak{I}$ be an ideal of $A$.
Let $\Sigma$ be a non-empty subset of $\ZZ_{\geq 0}^n$ such that
$\Sigma \cap \bigcup_{f \in \mathfrak{I}} \Supp(f) = \emptyset$.
Then $\length_A (A/\mathfrak{I}) \geq \#(\Sigma)$.
\end{Lemma}

\begin{proof}
It is sufficient to show that $\length_A (A/\mathfrak{I}) \geq \#(\Sigma')$ for any finite subset $\Sigma'$ of $\Sigma$, so that
we may assume that $\Sigma$ is a finite set.
We set $\Sigma = \{ \gamma_1, \ldots, \gamma_N \}$ such that $\gamma_i$'s are distinct and
$|\gamma_i| \leq |\gamma_j|$ for $1 \leq j \leq i \leq N$.
Let 
\[
\mathfrak{I}_i := \begin{cases}
\mathfrak{I} & \text{if $i=0$}, \\
\mathfrak{I} + X^{\gamma_1}A + \cdots + X^{\gamma_i}A & \text{if $1 \leq i \leq N$}.
\end{cases}
\]
Let us check that $\mathfrak{I}_{i-1} \subsetneq \mathfrak{I}_i$ for each $i=1, \ldots, N$.
Otherwise, there are $f_1, \ldots, f_{i-1} \in A$ and $f \in \mathfrak{I}$ such that $X^{\gamma_i} =  f + X^{\gamma_1}f_1 + \cdots + X^{\gamma_{i-1}}f_{i-1}$, 
so that, by \eqref{eqn:lem:key:estimate:01},
\[
\gamma_i \in \Supp(f) \cup \Supp( X^{\gamma_1}f_1) \cup \cdots \cup \Supp( X^{\gamma_{i-1}}f_{i-1}).
\]
Therefore, as $\gamma_i \not\in \Supp(f)$, there is $j$ such that
$1 \leq j < i$ and $\gamma_i \in \Supp( X^{\gamma_j}f_j)$, that is,
$\gamma_i = \gamma_j + \gamma$ for some $\gamma \in \Supp(f_j)$, and hence
$\gamma = 0$ because $|\gamma_i| \leq |\gamma_j|$. 
This is a contradiction.
Therefore one has the assertion of the lemma.
\end{proof}

\section{Differential operators and multiplicities}
From now on, we assume that the characteristic of $k$ is zero.
For $\gamma = (e_1, \ldots, e_n) \in \ZZ_{\geq 0}^n$, we set
\[
\partial^{\gamma} =\frac{\partial^{\gamma}}{\partial X^\gamma} 
:= \frac{\partial^{|\gamma|}}{\partial X_1^{e_1} \cdots \partial X_n^{e_n}}
\quad\text{and}\quad
\partial_{\gamma} := \frac{1}{e_1! \cdots e_n!}\partial^{\gamma}.
\]
Note that Leibnitz's rule can be written by
\begin{equation}
\label{eqn:Leibniz:rule:mult}
\partial_\gamma(f_1 \cdots f_r) = \sum_{\substack{\gamma_1, \ldots, \gamma_r \in \ZZ_{\geq 0}^n\\ \gamma_1 + \cdots + \gamma_r = \gamma}} \partial_{\gamma_1}(f_1) \cdots \partial_{\gamma_r}(f_r)
\end{equation}
in terms of the normalized differential operators $\partial_{\gamma}$'s.

Let $\Delta$ be a lower saturated subset of $\RR_{\geq 0}^n$ and
$\Sigma := \Delta \cap \ZZ_{\geq 0}^n$. Note that $\Sigma$ is a lower $\ZZ_{\geq 0}$-saturated subset of $\ZZ_{\geq 0}^n$.

\begin{Proposition}\label{prop:length:estimate:maximal:ideal}
Let $\mathfrak{m}$ be a maximal ideal of $A$ and $\mathfrak{I}$ be an ideal of $A$ such that $\mathfrak{m}$ is a minimal prime of $\mathfrak{I}$.
If $\partial_{\gamma}(\mathfrak{I}) \subseteq \mathfrak{m}$
for all $\gamma \in \Sigma$, then 
$\length_{A_\mathfrak{m}}((A/\mathfrak{I})_\mathfrak{m}) 
\geq \#(\Sigma)$.
\end{Proposition}

\begin{proof}
First of all, let us claim the following:

\begin{Claim}\label{claim:prop:length:estimate:maximal:ideal:01}
We may assume that $\sqrt{\mathfrak{I}} = \mathfrak{m}$. 
\end{Claim}

\begin{proof}
Let us see that $\partial_{\gamma}(\mathfrak{I}A_\mathfrak{m}) \subseteq \mathfrak{m}A_\mathfrak{m}$ 
for any $\gamma \in \Sigma$.
Indeed, for $f \in \mathfrak{I}$ and $a \in A_\mathfrak{m}$, by virtue of Leibnitz's rule \eqref{eqn:Leibniz:rule:mult},
\[
\partial_\gamma(fa) = \sum_{\substack{\gamma_1, \gamma_2 \in \ZZ_{\geq 0}^n\\ \gamma_1 +\gamma_2 = \gamma}} \partial_{\gamma_1}(f) \cdot \partial_{\gamma_2}(a),
\]
so that $\partial_\gamma(fa) \in \mathfrak{m}A_\mathfrak{m}$ by our assumptions. Therefore, if we set $\mathfrak{I}' = \mathfrak{I}A_\mathfrak{m} \cap A$, then $\sqrt{\mathfrak{I}'} = \mathfrak{m}$ and
$\partial_{\gamma}(\mathfrak{I}') \subseteq \mathfrak{m}A_\mathfrak{m} \cap A = \mathfrak{m}$.
Moreover, $\length_{A_\mathfrak{m}}((A/\mathfrak{I})_\mathfrak{m}) = \length_{A_\mathfrak{m}}((A/\mathfrak{I}')_\mathfrak{m})$, as required.
\end{proof}

From now on, we assume that $\sqrt{\mathfrak{I}} = \mathfrak{m}$, that is, $\Supp(A/\mathfrak{I}) = \{ \mathfrak{m} \}$.
In this case, note that $\length_{A_\mathfrak{m}}((A/\mathfrak{I})_\mathfrak{m}) = \length_{A}(A/\mathfrak{I})$.

First we assume that $k = A/\mathfrak{m}$. Then $\length_{A}(A/\mathfrak{I}) = \dim_k (A/\mathfrak{I})$.
Moreover there are $a_1, \ldots, a_n \in k$ such that $\mathfrak{m} = (X_1 - a_1, \ldots, X_n - a_n)$.
Note that $\partial/\partial X_i = \partial/\partial X'_i$,
where $X'_i = X_i - a_i$ for $i=1, \ldots, n$, so that,
replacing $X_i$ by $X'_i$, we may assume that
$a_1 = \cdots = a_n = 0$. By Lemma~\ref{lem:key:estimate},
it is sufficient to show $\Sigma \cap \bigcup_{f \in \mathfrak{I}} \Supp(f) = \emptyset$.
Otherwise, there are $\gamma \in \Sigma$ and $f \in \mathfrak{I}$ such that
$\gamma \in \Supp(f)$, so that we can set
\[
f = a_{\gamma} X^{\gamma} + \sum_{\gamma' \not= \gamma} a_{\gamma'} X^{\gamma'} \quad (a_{\gamma} \not= 0),
\]
and hence $\partial_{\gamma}(f) = a_{\gamma} + \sum_{\gamma' \not= \gamma} a_{\gamma'} \partial_{\gamma}(X^{\gamma'})$. Note that $\partial_{\gamma}(X^{\gamma'}) \in \mathfrak m$ for
all $\gamma'$ with $\gamma' \not= \gamma$. Thus $\partial_{\gamma}(f) \not\in \mathfrak m$,
which contradicts to $\gamma \in \Sigma$.

\bigskip
In general, let $\overline{k}$ be an algebraic closure of $k$. Then $\mathfrak{m}_{\overline{k}} := \mathfrak{m} \otimes_k \overline{k}$ and
$\mathfrak{I}_{\overline{k}} := \mathfrak{I} \otimes_k \overline{k}$ are ideals of $A_{\overline{k}} := A \otimes_k \overline{k} = \overline{k}[X_1, \ldots, X_n]$.
Moreover, if we set $r = [A/\mathfrak{m} : k]$, then there are distinct maximal ideals $\mathfrak{m}_1, \ldots, \mathfrak{m}_r$ 
and ideals $\mathfrak{I}_1, \ldots, \mathfrak{I}_r$ of $A_{\overline{k}}$ such that $\mathfrak{m}_{\overline{k}} = \mathfrak{m}_1 \cdots \mathfrak{m}_r$,
$\mathfrak{I}_{\overline{k}} = \mathfrak{I}_1 \cdots \mathfrak{I}_r$ and $\sqrt{\mathfrak{I}_i} = \mathfrak{m}_i$ for $i=1, \ldots, r$.

\begin{Claim}\label{claim:prop:length:estimate:maximal:ideal:02}
$\partial_\gamma(\mathfrak{I}_i) \subseteq \mathfrak{m}_i$ for any $\gamma \in \Sigma$ and $i=1, \ldots, r$.
\end{Claim}

\begin{proof}
Without loss of generality, we may assume that $i=1$. Fix $f_1 \in \mathfrak{I}_1$.
We prove it by induction on $|\gamma|$. The assertion for $|\gamma|= 0$ is obvious, so that we may assume that $\gamma \not= (0,\ldots, 0)$.
As $\mathfrak{m}_1 + \mathfrak{I}_i = A$ for $i=2, \ldots, r$, one can find
$f_i \in \mathfrak{I}_i$ such that $f_i \not\in \mathfrak{m}_1$. 
Since $\partial_\gamma$ is linear over $\overline{k}$, one has
$\partial_\gamma(\mathfrak{I}_{\overline{k}}) \subseteq \mathfrak{m}_{\overline{k}}$, so that
$\partial_\gamma(f_1 \cdots f_r) \in \mathfrak{m}_1$. On the other hand, by Leibnitz's rule \eqref{eqn:Leibniz:rule:mult},
\[
\partial_\gamma(f_1 \cdots f_r) = \partial_\gamma(f_1) f_2 \cdots f_r + \sum_{\substack{\gamma_1, \ldots, \gamma_r \in \ZZ_{\geq 0}^n\\\gamma_1 \not= \gamma,\ \gamma_1 + \cdots + \gamma_r = \gamma}} \partial_{\gamma_1}(f_1) \cdots \partial_{\gamma_r}(f_r)
\]
and $\sum_{\substack{\gamma_1, \ldots, \gamma_r \in \ZZ_{\geq 0}^n\\\gamma_1 \not= \gamma,\ \gamma_1 + \cdots + \gamma_r = \gamma}} \partial_{\gamma_1}(f_1) \cdots \partial_{\gamma_r}(f_r) \in \mathfrak{m}_1$
by the hypothesis of induction, so that $\partial_\gamma(f_1) f_2 \cdots f_r \in \mathfrak{m}_1$, and hence
$\partial_\gamma(f_1) \in \mathfrak{m}_1$ because $f_2 \cdots f_r \not\in \mathfrak{m}_1$.
\end{proof}

By Claim~\ref{claim:prop:length:estimate:maximal:ideal:02} together with the case where $k = A/\mathfrak{m}$, one has
\[
\dim_{\overline{k}} (A_{\overline{k}}/\mathfrak{I}_i) = \length_{A_{\overline{k}}}(A_{\overline{k}}/\mathfrak{I}_i) \geq \#(\Sigma)
\]
for $i=1, \ldots, r$.
Therefore,
{\allowdisplaybreaks \begin{align*}
\length_{A}(A/\mathfrak{I}) & = \frac{\dim_k (A/\mathfrak{I})}{[(A/\mathfrak{m}) : k]} = \frac{\dim_{\overline{k}} \left( (A/\mathfrak{I}) \otimes_k \overline{k}\right)}{[(A/\mathfrak{m}) : k]} = \frac{\dim_{\overline{k}} \left( A_{\overline{k}}/\mathfrak{I}_{\overline{k}}\right)}{[(A/\mathfrak{m}) : k]} \\
& = \frac{\dim_{\overline{k}} \left( \bigoplus_{i=1}^r A_{\overline{k}}/\mathfrak{I}_i \right)}{[(A/\mathfrak{m}) : k]}
= \frac{\sum_{i=1}^r \dim_{\overline{k}} (A_{\overline{k}}/\mathfrak{I}_i)}{[(A/\mathfrak{m}) : k]} \\
& \geq \frac{\sum_{i=1}^r \#(\Sigma)}{[(A/\mathfrak{m}) : k]} = \#(\Sigma),
\end{align*}}%
as required.
\end{proof}

Let $\mathfrak{p}$ be a prime ideal of $A$, $K$ be the fractional field of $A/\mathfrak{p}$ and $x_i$ be the image of $X_i$ in $A/\mathfrak{p}$ for $i=1, \ldots, n$.
Let 
\[
K_i := \begin{cases}
k(x_i, \ldots, x_n) & \text{if $i=1, \ldots, n$},\\
k & \text{if $i=n+1$}
\end{cases}
\]
for $i=1, \ldots, n+1$. Then
\[
K = K_1 \supseteq K_2 \supseteq \cdots \supseteq K_{n} \supseteq K_{n+1} = k.
\]
Let $\sigma_i$ be the transcendental degree of $K_i$ over $K_{i+1}$ for $i=1, \ldots, n$.
Note that $\sigma_1 + \cdots + \sigma_n = \dim (A/\mathfrak{p})$ and $\sigma_i \in \{ 0, 1 \}$ for all $i=1, \ldots, n$.
We set 
\[
\begin{cases}
\Upsilon := \{ i \in \{ 1, \ldots, n \} \mid \sigma_i = 0 \}, \\
(\RR^n)_\Upsilon := \{ (\xi_1, \ldots, \xi_n) \in \RR^n \mid \text{$\xi_j = 0$ for all $j \in \{ 1, \ldots, n \} \setminus \Upsilon$} \} \cong \RR^{\Upsilon}, \\
\Delta_\Upsilon = \Delta \cap (\RR^n)_\Upsilon, \quad \Sigma_\Upsilon := \Sigma \cap (\RR^n)_\Upsilon.
\end{cases}
\]
Note that $s:=\#(\Upsilon)$ is the codimension of $\Spec(A/\mathfrak{p})$ in $\Spec(A)$.

\begin{Corollary}
\label{coro:lower:estimate:mult}
Let $\mathfrak{I}$ be an ideal of $A$ such that $\mathfrak{p}$ is a minimal prime of $\mathfrak{I}$ and
$\partial_{\gamma}(\mathfrak{I}) \subseteq \mathfrak{p}$ for all $\gamma \in \Sigma$. Then
one has 
\[
\length_{A_\mathfrak{p}}((A/\mathfrak{I})_\mathfrak{p}) \geq
\#(\Sigma).
\]
In particular,  
$\length_{A_\mathfrak{p}}((A/\mathfrak{I})_\mathfrak{p}) \geq
\vol_{(\RR^n)_\Upsilon}(\Delta_\Upsilon)$,
where $\vol_{(\RR^n)_\Upsilon}(\Delta_\Upsilon)$ is the volume of
$\Delta_\Upsilon$ in $(\RR^n)_\Upsilon$.
\end{Corollary}

\begin{proof}
We set $\Upsilon_0 = \{ i_1, \ldots, i_s \}$ and $\{ 1, \ldots, n \} \setminus \Upsilon_0 = \{ j_1, \ldots, j_t \}$
($i_1 < \cdots < i_s$ and $j_1 < \cdots < j_t$). By our construction,
$x_{j_1}, \ldots, x_{j_t}$ are algebraically independent over $k$ and
$x_{i_1}, \ldots, x_{i_s}$ are algebraic over $k(x_{j_1}, \ldots, x_{j_t})$.
In particular, $k[X_{j_1}, \ldots, X_{j_t}] \setminus \{ 0 \} \subseteq A \setminus \mathfrak{p}$ and 
$K$ is finite over $k(x_{j_1}, \ldots, x_{j_t})$.

Let $A_S$ be the localization of $A$ with respect to $S = k[X_{j_1}, \ldots, X_{j_t}] \setminus \{ 0 \}$, that is,
$A_S = k(X_{j_1}, \ldots, X_{j_t})[X_{i_1}, \ldots, X_{i_s}]$.
By the above observation,
$\mathfrak{p}A_S$ gives rise to a maximal ideal of $A_S$ and $\mathfrak{p}A_S$ is a minimal prime of $\mathfrak{I}A_S$. Moreover, as
$\partial_\gamma(\mathfrak{I}) \subseteq \mathfrak{p}$
for $\gamma \in \Sigma_\Upsilon$, 
one can see $\partial_\gamma(\mathfrak{I} A_S) \subseteq \mathfrak{p}A_S$.
Therefore, the first assertion of the corollary follows from Proposition~\ref{prop:length:estimate:maximal:ideal}.
The  second assertion follows from Proposition~\ref{prop:lower:saturated}.
\end{proof}

\begin{Example}\label{example:mult:estimate}
Under the assumptions of Corollary~\ref{coro:lower:estimate:mult},
we introduce a special closed lower saturated
subset $\Delta$ of $\RR_{\geq 0}^n$.
Let $\pmb{d} = (d_1, \ldots, d_n)$ be a sequence of positive number. For $\varphi = (\xi_1, \ldots, \xi_n) \in \RR_{\geq 0}^n$,
we set $|\varphi|_{\pmb{d}} := \xi_1/d_1 + \cdots + \xi_n/d_n$. 
For a non-negative number $\epsilon$, we set
$\Delta := \{ \varphi \in \RR_{\geq 0}^n \mid |\varphi|_{\pmb{d}} \leq \epsilon \}$. Note that $\Delta$ is lower saturated and
\[
\vol_{(\RR^n)_\Upsilon}(\Delta_\Upsilon) = (\epsilon^s/s!) \prod_{j \in \Upsilon} d_j =
 (\epsilon^s/s!) d_{1}^{1-\sigma_1} \cdots d_{n}^{1-\sigma_n}.
\]
Therefore,
by Corollary~\ref{coro:lower:estimate:mult}, 
if $\partial_\gamma(\mathfrak{I}) \subseteq \mathfrak{p}$ for all $\gamma \in \Sigma$,
then 
\begin{equation}\label{eqn:example:mult:estimate:01}
\length_{A_\mathfrak{p}}((A/\mathfrak{I})_\mathfrak{p}) \geq (\epsilon^s/s!) d_1^{1 - \sigma_1} \cdots d_n^{1 - \sigma_n}.
\end{equation}

\medskip
We assume that there are integers $\alpha_1, \ldots, \alpha_m, \alpha_{m+1}$ such that
\[
1 = \alpha_1 < \alpha_2 < \cdots < \alpha_m < \alpha_{m+1}=n+1
\]
and $d_{\alpha_i} = \cdots = d_{\alpha_{i+1} - 1}$
for all $i=1, \ldots, m$. We set
\[
\begin{cases}
\mathfrak{d}_i := d_{\alpha_i},\quad
\mathfrak{n}_i := \alpha_{i+1} - \alpha_i, \\
\delta_i := \text{the trancendental degree of $K_{\alpha_i}$ over $K_{\alpha_{i+1}}$}.
\end{cases}
\]
for $i=1, \ldots, m$.
In this case, \eqref{eqn:example:mult:estimate:01} asserts that
\begin{equation}\label{eqn:example:mult:estimate:02}
\length_{A_\mathfrak{p}}((A/\mathfrak{I})_\mathfrak{p}) \geq (\epsilon^s/s!) \prod_{i=1}^m \mathfrak{d}_{i}^{\mathfrak{n}_i - \delta_i}.
\end{equation}
Indeed, as $\sum_{j=\alpha_i}^{\alpha_{i+1}-1} \sigma_j = \delta_i$,
one has 
\[
\prod_{j=1}^n d_j^{1 - \sigma_j} = \prod_{i=1}^m \prod_{j=\alpha_i}^{\alpha_{i+1}-1} \mathfrak{d}_{i}^{1 - \sigma_j} = \prod_{i=1}^m \mathfrak{d}_{i}^{\mathfrak{n}_i - \delta_i},
\]
as desired.
\end{Example}

\appendix
\renewcommand{\theTheorem}{\Alph{section}.\arabic{Theorem}}
\renewcommand{\theClaim}{\Alph{section}.\arabic{Theorem}.\arabic{Claim}}
\renewcommand{\theequation}{\Alph{section}.\arabic{equation}}
\section{Lebesgue measurability of lower saturated subset}

In this appendix, let us consider the measurability of lower saturated subsets of $\RR_{\geq 0}^n$, that is,
one has the following:

\begin{Proposition}\label{prop:measurable:Delta}
\label{prop:measure:lower:saturated:subset}
Let $\Delta$ be a lower saturated subset of $\RR_{\geq 0}^n$.
Let $\overline{\Delta}$ be the closure of $\Delta$ and $\Delta^{\circ}$ be
the set of all interior points of $\Delta$.
Then $\vol(\overline{\Delta} \setminus \Delta^{\circ}) = 0$. In particular,
$\Delta$ is Lebesgue measurable. 
\end{Proposition}

\begin{proof}
For a positive integer $N$, we set $\Delta_{N} := \Delta \cap [0, N]^n$.
Then $\Delta_{N}$ is lower saturated.
Moreover, we set
\[
{\displaystyle J_{N, m} := \left\{\left. \frac{j}{m} \ \right|\  0 \leq j \leq mN,\ j \in \ZZ \right\}}\quad\text{and}\quad \Sigma_{N, m} := \Delta \cap (J_{N, m})^n
\]
for a positive integer $m$. Further, we set
\[
\Delta_{N, m} := \bigcup_{\varphi \in \Sigma_{N, m}} \mathbb{I}_\varphi\quad\text{and}\quad 
\Delta'_{N, m} := \kern-1em \bigcup_{(a_1, \ldots, a_n) \in \Sigma_{N, m}} \left[a_1, a_1+ \frac{1}{m}\right] \times \cdots \times \left[a_n, a_n + \frac{1}{m}\right].
\]
Note that $\Delta_{N,m} \subseteq \Delta_N \subseteq \Delta'_{N, m}$, and $\Delta_{N,m}$ and $\Delta'_{N, m}$ are closed.
Then one has the following:

\begin{Claim}\label{claim:prop:measurable:Delta:01}
\begin{enumerate}[(1)]
\item 
$\vol(\Delta'_{N,m} \setminus \Delta_{N,m}) \leq n(1 + N)^{n-1}/m$.

\item
$\vol(\overline{\Delta} \cap [0, N]^n) \leq \vol(\Delta'_{N,m})$.

\item
$\vol(\Delta_{N,m}) \leq \vol(\Delta^{\circ} \cap [0, N]^n)$.

\item
$\vol((\overline{\Delta} \setminus \Delta^{\circ}) \cap [0, N]^n) = 0$.
\end{enumerate}
\end{Claim}

\begin{proof}
(1) We set
\[
\Delta''_{N,m} := \Delta'_{N, m} - \left(\frac{1}{m}, \ldots, \frac{1}{m}\right) = \kern-1em \bigcup_{(a_1, \ldots, a_n) \in \Sigma_{N, m}} \left[a_1- \frac{1}{m}, a_1\right] \times \cdots \times \left[a_n- \frac{1}{m}, a_n\right].
\]
Note that
$\Delta_{N, m} \subseteq \Delta''_{N, m}$.
Moreover, if we set
\[
\partial S := \{ (a_1, \ldots, a_n) \in S \mid \exists\, i \in \{ 1, \ldots, n \},\ a_i = 0 \}
\]
for a subset $S$ of $(J_{N, m})^n$, then
\[
\Delta''_{N, m} \setminus \Delta_{N,m} \subseteq \bigcup_{(a_1, \ldots, a_n) \in \partial\Sigma_{N,m}} \left[a_1- \frac{1}{m}, a_1\right] \times \cdots \times \left[a_n- \frac{1}{m}, a_n\right].
\]
Therefore, one has
\begin{align*}
\vol(\Delta'_{N,m} \setminus \Delta_{N, m}) & = \vol(\Delta'_{N,m}) - \vol(\Delta_{N,m}) \\
& = \vol(\Delta''_{N, m}) - \vol(\Delta_{N,m}) = \vol(\Delta''_{N,m} \setminus \Delta_{N,m}) \\
& \leq \frac{\#(\partial\Sigma_{N, m})}{m^n} \leq \frac{\#(\partial (J_{N, m})^n)}{m^n}.
\end{align*}
On the other hand, since
\begin{multline*}
\partial (J_{N,m})^n =
\{ (a_1, \ldots, a_{n-1}, 0)  \mid (a_1, \ldots, a_{n-1}) \in (J_{N,m})^{n-1} \} \\
\amalg
\{ (a_1, \ldots, a_{n-1}, a_n)  \mid \text{$(a_1, \ldots, a_{n-1}) \in \partial (J_{N,m})^{n-1}$ and $a_n \not= 0$} \},
\end{multline*}
one has 
\begin{align*}
\#(\partial(J_{N,m})^n) & = (1 + mN)^{n-1} + (mN)\#(\partial(J_{N,m})^{n-1}) \\
& \leq m^{n-1}(1 + N)^{n-1} + m(1+N) \#(\partial(J_{N,m})^{n-1}),
\end{align*}
so that one can see that
$\#(\partial(J_{N,m})^n) \leq nm^{n-1}(1 + N)^{n-1}$ by induction on $n$, as required.

\medskip
(2) First as $\vol([0, N]^n \setminus [0, N)^n) = 0$, 
one has $\vol(\overline{\Delta} \cap [0, N]^n) = \vol(\overline{\Delta} \cap [0, N)^n)$.
Next let us see that $\overline{\Delta} \cap [0, N)^n\subseteq \overline{\Delta_N}$.
Indeed, if $\varphi \in \overline{\Delta} \cap [0, N)^n$, then
there is a sequence $\{ \varphi_m \}_{m=1}^{\infty}$ such that $\varphi_m \in \Delta$ for all $m$ and $\varphi = \lim_{m\to\infty} \varphi_m$.
As $\varphi \in [0, N)^n$, there is a positive integer $M$ such that $\varphi_m \in [0, N)^n$ for $m \geq M$, so that $\varphi \in \overline{\Delta_N}$, as desired.
Therefore, we obtain  $\vol(\overline{\Delta} \cap [0, N]^n) \leq \vol(\overline{\Delta_N})$.

On the other hand, as $\Delta_N \subseteq \Delta'_{N,m}$ and
$\Delta'_{N,m}$ is closed, one has $\overline{\Delta_N} \subseteq \Delta'_{N,m}$, so that
$\vol(\overline{\Delta_N}) \leq \vol(\Delta'_{N,m})$, and hence the assertion of (2) follows.

\medskip
(3) As $\mathbb{I}^{\circ}_{\varphi} \subseteq \Delta^{\circ}$ for $\varphi \in \Sigma_{N, m}$, one obtains
$\bigcup_{\varphi \in \Sigma_{N, m}} \mathbb{I}^{\circ}_{\varphi} \subseteq \Delta^{\circ} \cap [0, N]^n$, and hence
$\vol\left(\bigcup_{\varphi \in \Sigma_{N, m}} \mathbb{I}^{\circ}_{\varphi}\right)
\leq \vol(\Delta^{\circ} \cap [0, N]^n)$. On the other hand, note that
\[
\Delta_{N, m} \setminus 
\bigcup\nolimits_{\varphi \in \Sigma_{N, m}} \mathbb{I}^{\circ}_{\varphi} =
\left(\bigcup\nolimits_{\varphi \in \Sigma_{N, m}} \mathbb{I}_{\varphi}\right) \setminus 
\left(\bigcup\nolimits_{\varphi \in \Sigma_{N, m}} \mathbb{I}^{\circ}_{\varphi}\right) \subseteq 
\bigcup\nolimits_{\varphi \in \Sigma_{N, m}} (\mathbb{I}_{\varphi} \setminus \mathbb{I}^{\circ}_{\varphi}),
\]
so that
\begin{align*}
 0 & \leq \vol(\Delta_{N, m}) - \vol\left(
\bigcup\nolimits_{\varphi \in \Sigma_{N, m}} \mathbb{I}^{\circ}_{\varphi}\right) = \vol\left(\Delta_{N, m} \setminus 
\bigcup\nolimits_{\varphi \in \Sigma_{N, m}} \mathbb{I}^{\circ}_{\varphi}\right) \\
& \leq \vol\left(
\bigcup\nolimits_{\varphi \in \Sigma_{N, m}} (\mathbb{I}_{\varphi} \setminus \mathbb{I}^{\circ}_{\varphi})\right) \leq \sum\nolimits_{\varphi \in \Sigma_{N, m}} \vol(\mathbb{I}_{\varphi} \setminus \mathbb{I}^{\circ}_{\varphi}) = 0.
\end{align*}
Therefore, $\vol(\Delta_{N,m}) \leq \vol(\Delta^{\circ} \cap [0, N]^n)$, as required.

\medskip
(4) 
By (1), (2) and (3), one has
\begin{align*}
0 & \leq \vol((\overline{\Delta} \setminus \Delta^{\circ}) \cap [0, N]^n) =
\vol(\overline{\Delta} \cap [0, N]^n) - \vol(\Delta^{\circ} \cap [0, N]^n) \\
& \leq \vol(\Delta'_{N,m}) - \vol(\Delta_{N,m}) \leq n(1 + N)^{n-1}/m,
\end{align*}
so that the assertion of (4) follows because $m$ is an arbitrary positive integer.
\end{proof}

By (4) in Claim~\ref{claim:prop:measurable:Delta:01},
one can see that $\vol(\overline{\Delta} \setminus \Delta^{\circ}) = 0$ because 
\[
\bigcup_{N\geq 1} (\overline{\Delta} \setminus \Delta^{\circ}) \cap [0, N]^n = \overline{\Delta} \setminus \Delta^{\circ}.
\]
In particular, since
$\Delta = \Delta^{\circ} \cup (\Delta \cap (\overline{\Delta} \setminus \Delta^{\circ}))$ and
the Lebesgue measure is complete,
the measurability of $\Delta$ follows.
\end{proof}

\end{document}